\documentclass[11pt]{article}
\usepackage[T1]{fontenc}
\usepackage{amsmath,amssymb,amsthm,amsfonts}
\usepackage{algorithm, algpseudocode, algorithmicx}

\usepackage[svgnames]{xcolor}
\usepackage[colorlinks,urlcolor=black,citecolor={green!40!red!50!black}]{hyperref}
\usepackage[authoryear, square, semicolon, compress]{natbib}
\usepackage{nicefrac}

\DeclareMathOperator{\mean}{mean}
\DeclareMathOperator{\summ}{sum}
\DeclareMathOperator{\minimize}{minimize}
\newtheorem{lemma}{Lemma}

\newcommand{\rom}[1]{\uppercase\expandafter{\romannumeral #1\relax}}

\title{Gradient Projection for Solving Quadratic Programs with Standard Simplex Constraints}
\author{Youwei Liang\footnote{College of Mathematics and Informatics, South China Agricultural University, Guangzhou 510642, China. {\sl Email}: liangyouwei1@gmail.com}}
\date{}

\begin{document}
\maketitle
\begin{abstract}
An important method to optimize a function on standard simplex is the active set algorithm, which requires the gradient of the function to be projected onto a hyperplane, with sign constraints on the variables that lie in the boundary of the simplex. We propose a new algorithm to efficiently project the gradient for this purpose. Furthermore, we apply the proposed gradient projection method to quadratic programs (QP) with standard simplex constraints, where gradient projection is used to explore the feasible region and, when we believe the optimal active set is identified, we switch to constrained conjugate gradient to accelerate convergence. Specifically, two different directions of gradient projection are used to explore the simplex, namely, the projected gradient and the reduced gradient. We choose one of the two directions according to the angle between the directions. Moreover, we propose two conditions for guessing the optimal active set heuristically. The first condition is that the working set remains unchanged for many iterations, and the second condition is that the angle between the projected gradient and the reduced gradient is small enough. Based on these strategies, a new active set algorithm for solving quadratic programs on standard simplex is proposed. 
\end{abstract}

\section{Introduction}
Given a function $f \colon \mathbb{R}^{n} \to \mathbb{R}^{m}$, suppose we want to minimize $f$ on a constant-sum simplex.
\begin{align}
	\underset{\alpha}\minimize\quad &f(\alpha) \\
	\text{subject to}\quad & e^\top \alpha = c, \  \alpha \geq 0 \label{eq:constraint}
\end{align}
where $e$ is an all-one vector and $c$ is a constant. When $c=1$, the constraint is a standard simplex (a.k.a. probability simplex). Intrinsically this leads to combinatorial optimization since we need to decide which elements in $\alpha$ should be $0$ and which should be greater than $0$. 
A special case is that $f$ is a quadratic function, which arises in various applications such as clustering \citep{liang2019consistency} and will be discussed in Section~\ref{sec:qp}. 
A popular iterative approach to solve the problem is the active set method \citep{wright2006} with gradient projection \citep{birgin2000nonmonotone,Cristofari2020,dai2006new,di2018two}. An active set is a set determining which elements in $\alpha$ are fixed to $0$ and which elements are free variables. With an active set, we optimize $f$ with respect to free variables without applying inequality constraints on them. A general optimization method working with active set method is gradient descent, which seeks to minimize $f$ by taking a step along the opposite direction of the gradient of $f$. In constrained optimization, however, directly taking a step with gradient may cause $\alpha$ to violate the constraints. Therefore, we need to project the gradient to a space where the constraints hold. 
Let $g=\nabla f$ be the gradient of $f$, then we want to take a step along the projected gradient $\tilde{g}$, which should be as close to $g$ as possible while the constraints \eqref{eq:constraint} are satisfied. Then we have $e^\top (\alpha-u \tilde{g}) = c, \ (\alpha-u \tilde{g}) \geq 0$, where $u > 0$ is the step size. Under the framework of active set method, some elements in $\alpha$ are fixed to $0$ while other elements are free variables that can be changed. Therefore, $\tilde{g}$ must satisfy $e^\top \tilde{g}=0$ and $\tilde{g}_{i} \leq 0$ for all $i \in G$ where $G$ is a set defined as $G = \{i \mid \alpha_{i}=0\}$.
Then we seek to project the gradient $g$ onto the hyperplane $e^\top \alpha=0$ with sign constraints on some elements of $g$, which is formulated as the gradient projection problem.
\begin{align}
	\underset{x}\minimize \quad &h(x) = \|x - g\|^2 \label{P:1} \\
	\text{subject to} \quad &e^\top x=0 \label{equality} \\
	&x_i \leq 0, \quad i \in G. \label{inequality}
\end{align}
The solution to this problem is the projected gradient $\tilde{g}$. 
Some similar projection problems where the inequalities \eqref{inequality} are imposed on all variables have been addressed by many authors \citep{duchi2008efficient,chen2011projection,wang2013projection,wang2015projection,condat2016fast}, while how to solve the problem with inequality \eqref{inequality} imposed on partial variables is not investigated. In this paper, we analyze the properties of the solution to Problem~\eqref{P:1} and present an efficient algorithm based on our analysis.

\section{Gradient Projection}
\subsection{Theoretical Properties of the Solution}
Without loss of generality, we assume that that elements in $g$ is in descending order such that $g_1 \geq g_2 \geq \dots \geq g_n$, where $n$ is the number of elements in $g$. Let $I = \{1, \dots, n\}$.
An important property of the solution to the standard simplex projection is that it preserves the order of the elements in the vector being projected \citep[Lemma~1]{duchi2008efficient}. However, when the vector being projected has sign constraints on only some of its elements, Lemma~1 in \citep{duchi2008efficient} no long holds. Instead, we have the following lemma.
\begin{lemma} \label{lemma:1}
	Let $x^*$ be the optimal solution of problem~\eqref{P:1}. Let $I^+$ and $I^-$ denote the index set of the non-negative and non-positive elements of $x^*$ respectively. If $i, j \in I^+, \ i \leq j$, we have $x_i^* \leq x_j^*$. Similarly, if $i, j \in I^-, \ i \leq j$, we have $x_i^* \leq x_j^*$. 
\end{lemma}
\begin{proof}
	Suppose $x$ is the minimizer of problem \eqref{P:1} and $g_i > g_j$. Suppose $ x_i < x_j \leq 0$, i.e., $x_i, x_j \in I^-$. Switch $x_i$ and $x_j$ to get a new solution $\tilde{x}$ where $\tilde{x}_i = x_j, \tilde{x}_j = x_i$ and $\tilde{x}_k = x_k$ for $k \neq i,j$. Note that $\tilde{x}$ satisfies the constraints \eqref{equality} and \eqref{inequality}. Then
	\begin{align*}
		\Delta h(x)
		=& \|\tilde{x} - g\|^2 - \|x - g\|^2 \\
		=& (x_j - g_i)^2 + (x_i - g_j)^2 - (x_i - g_i)^2 - (x_j - g_j)^2 \\
		=& 2(x_j-x_i)(g_j-g_i) \\
		<& 0
	\end{align*}
	This contradicts that $x$ is the minimizer. Thus, we conclude $x_i \geq x_j$.
	The analysis for $x_i, x_j \in I^+$ is similar and thus omitted here.
	Note that if $g_i = g_j$, $\Delta f(x) = 0$ and thus there might be multiple optimal solutions. We \emph{choose} to adopt the optimal solution that obeys Lemma~\ref{lemma:1}.
\end{proof}
Noting that minimizing $\|x - g\|^2$ is equivalent to minimizing $x^\top x - 2x^\top g$, we construct a Lagrangian function $\displaystyle {\mathcal {L}}(x,\lambda, \mu )= x^\top x - 2x^\top g - \lambda e^\top x + \mu^\top x$, where $\mu$ is a vector defined as $\mu_i = 0$ if $i \in G^c = \{k \mid 1 \leq k \leq n, k \notin G\} $. 
KKT conditions imply
\begin{align}
	\frac{\partial \mathcal{L}}{\partial x} &= 2x-2g- \lambda e + \mu = 0 \label{E:lagr1} \\
	e^\top x &= 0 \label{E:lagr2} \\
	\mu &\geq 0 \label{E:mu} \\
	\mu_i x_i &= 0, \quad i \in G \label{E:dual}
\end{align}
Let $J$ denote the index set of the non-zero elements of $x^*$. For any $j \in J, \, x_j^* \neq 0$, by \eqref{E:dual} and the definition of $\mu$ (note that it is possible that $j \in G^c$), we have $\mu_j = 0$. By \eqref{E:lagr1}, $2x_j^* - 2g_j - \lambda = 0$. Since $\sum_{j \in J} x_j^* = 0$, $\sum_{j \in J} (2x_j^* - 2g_j - \lambda) = \sum_{j \in J} (- 2g_j - \lambda) = 0$. Let $\bar{g}_J = \sum_{j \in J} g_j / m$, i.e., the average of $g_J = \{g_j \mid j \in J\}$. Thus
\begin{equation}
	\lambda = \frac{-2 \sum_{j \in J} g_j}{m} = -2\bar{g}_J \label{E:lambda}
\end{equation} 
where $m$ is the number of elements in $J$. 
For all $i \in G \setminus J$, we have $x_i = 0$. Then by \eqref{E:lagr1} we have
\begin{equation}
	\mu_i = 2 g_i + \lambda = 2 g_i - \frac{2 \sum_{j \in J} g_j}{m} = 2(g_i - \bar{g}_J) \label{E:mu2}
\end{equation}
For all $i \in J \cup G^c$, we have $\mu_i = 0$. Then by \eqref{E:lagr1} we have
\begin{equation}
	x_i^* = g_i + \frac{\lambda}{2} = g_i - \frac{\sum_{j \in J} g_j}{m} = g_i - \bar{g}_J \label{E:x}
\end{equation}
Let $\bar{g}_I = \sum_{i \in I} g_i / n$, i.e., the average of $g_I$ (which is also $g$).
Let $x$ be defined as $x_i = g_i - \bar{g}_I$ for $i \in I$. If for $i \in G, \, x_i \leq 0$, then $x$ is the optimal solution since it is the optimal solution to the problem without inequality constraints \eqref{inequality} (this can be checked by the optimal conditions for the equality-constrained problem).

If for some $i \in G, \, x_i > 0$, let $A = \{i \in G \mid x_i > 0 \}$ and $a = \max (A)$. Note that 
\begin{equation}
	x_a = g_a - \bar{g}_I > 0 \label{E:x_a}
\end{equation}
Since $a \in G$, only one of the two situation can happen: $x_a^* = 0$ or $x_a^* < 0$. 

Suppose $x_a^* < 0$, then $a \in J$ and by \eqref{E:x} we have $ g_a - \bar{g}_J < 0$. Thus $\bar{g}_J > \bar{g}_I$. Let $F = \{i \mid 1 \leq i \leq n, \ g_i < \bar{g}_I \}$, i.e., the index set of $g_i$'s which are smaller than the average of $g_I$. $\bar{g}_J > \bar{g}_I$ implies that for some $f \in F$, $f \in J^c = I \setminus J$. Thus, $x_f^* = 0$. Note that $f, a \in I^-$. By \eqref{E:x_a} and definition of $F$, $g_f < \bar{g}_I < g_a$, by Lemma~\ref{lemma:1}, $x_f^* \leq x_a^*$. But $x_f^* = 0 > x_a^*$ causes a contradiction. Thus it is impossible that $x_a^* < 0$ and only $x_a^* = 0$ can be true. In the analysis we set $a = \max (A)$ only for ease of introducing our algorithm. In fact for all $i \in A$, $x_i^* = 0$.

Since $x_a^* = 0$, we can remove $g_a$ from $g$ and construct a reduced problem. Formally, let $I' = I \setminus \{a\}$ and $G' = G \setminus \{a\}$ be the reduced index sets. The reduced problem is
\begin{align}
\min_x \quad &\|x_{I'} - g_{I'}\|^2 \\
\text{s.t.} \quad &e^\top x_{I'}=0 \\
&x_i \leq 0, \quad i \in G' \label{inequality2}
\end{align}
Repeating the same analysis for $g_{I'}$ and $x_{I'}$. Either the inequality constraints \eqref{inequality2} are satisfied or a zero element in $x_{I'}^*$ is determined. Repeat the procedures until $x^*$ is found.

\subsection{An Gradient Projection Algorithm}
In this section, we present an algorithm for solving problem \eqref{P:1} based on our analysis.
\begin{algorithm}[H]
\begin{algorithmic}[1]
	\renewcommand{\algorithmicrequire}{\textbf{Input:}}
	\renewcommand{\algorithmicensure}{\textbf{Output:}}
	\Require $g$, $G$, $n$
	\Ensure  $x$
	\State Sort $g$ into descending order such that $g_1 \geq g_2 \geq \dots \geq g_n$. Reset the indices in $G$ to match the indices of sorted $g$. And sort $G$ into ascending order.
	\State $a = \mean(g), \ s = \summ(g)$ \Comment{the mean and sum of $g$}
	\State $m = |G|$ \Comment{number of inequality constraints}
	\State $H = \{\}$
	\For {$i= 1, \dots, m$}
	\edef\savedSixteenTwo{\the\fontdimen16\textfont2 }
	\fontdimen16\textfont2=4pt
	\If {$g_{G_i} > a$}
	\State $s \gets s - g_{G_i}$
	\fontdimen16\textfont2=\savedSixteenTwo
	\State $ a \gets s / (n - i)$
	\State $H \gets H \cup \{G_i\}$ \Comment{add $G_i$ to the index set $H$}
	\Else 
	\State Break
	\EndIf
	\EndFor
	\State $x_i \gets g_i - a$ for $i \in I \setminus H$
	\State $x_i \gets 0$ for $i \in H$
	\State Reorder the elements in $x$ to match the original order of $g$ before sorting
\end{algorithmic}
\caption{Gradient Projection onto Simplex}
\label{alg:proj grad}
\end{algorithm}

\edef\savedSixteenTwo{\the\fontdimen16\textfont2 }
\fontdimen16\textfont2=4pt
In the for loop (line 5 -- 13) in Algorithm~\ref{alg:proj grad}, since $g_{G_i} > a$ and $g_{G_i}$ is removed from the sum $s$, $a$ is decreasing during the procedure.\fontdimen16\textfont2=\savedSixteenTwo When the algorithm terminates,  $J = I \setminus H$ and $a$ equals $\bar{g}_J$.
Thus for $i \in H$, $g_i > a = \bar{g}_J$, and the KKT multiplier for $x_i^*$ is $\mu_i = 2(g_i - \bar{g}_J) > 0$, satisfying the KKT condition \eqref{E:mu}.

In Algorithm~\ref{alg:proj grad}, the computation bottleneck lies in the sorting of the input vector, and thus the time complexity is the same as that of the sorting algorithm. Many sorting algorithms has $O(n \log n)$ time complexity, and thus Algorithm~\ref{alg:proj grad} can be run efficiently in $O(n \log n)$ time.

\section{Quadratic Programs with Standard Simplex Constraints} \label{sec:qp}
Quadratic program (QP) on standard simplex is the following optimization problem.
\begin{align} \label{sub-problem 1}
\underset{\alpha}\minimize \quad & q({\alpha}) = \nicefrac{1}{2}\; {\alpha}^\top {H} {\alpha} - {\alpha}^\top {c} \\
\text{subject to} \quad &{\alpha}^\top  {1}=1, \ {\alpha} \geq 0\label{eq:alpha1}
\end{align}
where $H$ is a symmetric matrix but not necessarily positive semidefinite, and $c$ is a vector. Let $\mathcal{G}=\{{\alpha} \geq 0 \mid {\alpha}^\top {1}=1\}$ denote the feasible region. 
The active set method starts by making a guess of the optimal active set $\mathcal{A}^*$, that is, the set of constraints that are satisfied as equalities at a solution \citep{wright2006}. In our QP, the active set is $\mathcal{A}=\{i \mid \alpha_{i} = 0\}$. We call our guess of the optimal active set the working set and denote it by $\mathcal{W}$. We then solve the QP in which the constraints in the working set $\mathcal{W}$ are imposed as equalities and the constraints not in $\mathcal{W}$ are ignored. 
When combined with gradient descent, the basic procedure in the active set algorithm (ASA) is to generate a sequence of feasible points on the working set $\mathcal{W}$ until a stationary point is found. 
Various strategies for generating such points are proposed in the literature \citep{dembo1984minimization,more1991solution,dai2005projected,dai2006cyclic,hager2005conjugate,hager2006new,di2018two,Cristofari2020}. 
Starting from an initial guess of the optimal active set, we use projected gradient method to explore the feasible region $\mathcal{G}$. If the boundary of $\mathcal{G}$ is encountered, we modify $\mathcal{W}$ to include the boundary, i.e., adding a new index to $\mathcal{W}$. We will also use some strategies to remove an index from $\mathcal{W}$, following similar strategies in \citep{dembo1984minimization,hager2006new}. 
When we feel the optimal active set is identified, we switch to conjugate gradient method which has faster convergence rate than gradient descent \citep{hager2006new}. 

\subsection{Gradient Projection}
This section discusses how to explore $\mathcal{G}$ using gradient projection without violating the constraints. 
Let $ d$ denote the gradient of $q({\alpha})$ with respect to ${\alpha}$. Using gradient descent, we can obtain the next iteration by setting ${\alpha} \gets {\alpha} - u  d$ where $u > 0$ is the step size, but the new iteration may violate constraints \eqref{eq:alpha1}. To make sure that each iteration lies within the feasible region, we need to project the gradient $ d$ onto the invariant subspace with respect to constraints \eqref{eq:alpha1}. That is, after taking a step, constraints \eqref{eq:alpha1} still hold, which means the projected gradient $\tilde{ g}$ satisfies $({\alpha} - u \tilde{ g})^\top  {1}=1$ and ${\alpha} - u \tilde{ g} \geq 0$. 
Following the idea from \cite{dembo1984minimization}, we only select some elements in $ d$ to project. The binding set is defined as $\mathcal{B} = \{i \mid i \in \mathcal{W}, d_i \geq 0\}$. Then the constraints associated with $\mathcal{B}$ are the constraints whose associated Lagrange multiplier estimates have the correct sign \citep{dembo1984minimization}. 
The idea from \cite{dembo1984minimization} is to introduce two directions, the reduced gradient $ g^R$ and the projected gradient $ g^P$, defined as follows. 
\begin{gather} \label{eq:twodirection}
g_{i}^R = 
\begin{cases}
0 & \text{if $i \in \mathcal{W}$}\\
d_{i} & \text{if $i \notin \mathcal{W}$}
\end{cases}, \quad
g_{i}^R = 
\begin{cases}
0 & \text{if $i \in \mathcal{B}$}\\
d_{i} & \text{if $i \notin \mathcal{B}$}
\end{cases}.
\end{gather}
Let $ g$ denote $ g^R$ or $ g^P$, and let $\mathcal{T}$ be the corresponding set $\mathcal{W}$ or $\mathcal{B}$. The gradient projection can be formulated as the following optimization problem on standard simplex:
\begin{align}
\min_{\tilde{ g}} \quad &\|\tilde{ g} -  g\|^2 \label{P:3} \\
\text{s.t.} \quad & 1^\top \tilde{ g}=0 \label{eq:eq} \\
&\tilde{g}_i = g_i, \; i \in \mathcal{T}\\
&\tilde{g}_i \leq 0, \; i \in G \label{eq:ineq}
\end{align}
where $G = \{i \mid \alpha_{i}=0\}$.\footnote{$G$ is not necessarily the same as $\mathcal{W}$.} The gradient projection problem \eqref{P:3} can be efficiently solved by Algorithm~\ref{alg:proj grad} (note that we only need to project the elements not in $\mathcal{T}$). Let $\tilde{ g}^R$ and $\tilde{ g}^P$ be the projection of $ g^R$ and $ g^P$, respectively. In Algorithm~\ref{alg:asa}, we show how to choose between $\tilde{ g}^R$ and $\tilde{ g}^P$. 
With $\tilde{ g}^R$ or $\tilde{ g}^P$ chosen, we take a step along it and the objective is a quadratic function w.r.t. the step size $u$. Then we can easily find the optimal step size on an interval $[0, u_{max}]$ to reduce the objective (i.e.,using exact line search) and ensure $\alpha_i \geq 0$ at the same time. 

\subsection{Constrained Conjugate Gradient}
When we feel the optimal active set is identified, we would like to switch to an unconstrained optimization algorithm to accelerate convergence, since gradient projection may converge very slowly. An approach is to use conjugate gradient for acceleration \citep{hager2006new}.  
Since we are using conjugate gradient on the working set $\mathcal{W}$, and $\alpha_{i}$ in $\mathcal{W}$ must remain $0$, we can construct a new QP on the free variables (i.e., variables not in $\mathcal{W}$), on which we run conjugate gradient. Let $\tilde{{\alpha}}$ be a vector composed of the free variables of ${\alpha}$. Suppose the new QP is of the following form $\tilde{q}(\tilde{{\alpha}}) = \nicefrac{1}{2}\ \tilde{{\alpha}}^\top \tilde{ H} \tilde{{\alpha}} - \tilde{{\alpha}}^\top \tilde{ c}$ with $m$ free variables ($\tilde{{\alpha}} \in \mathbb{R}^m$). 
Then $\tilde{ H}$ is the matrix obtained from the original Hessian by taking those rows and columns whose indices correspond to free variables; similarly, $\tilde{ c}$ is obtained from $ c$ by taking the components whose indices correspond to free variables (this is because the non-free variables are 0). 

To run conjugate gradient on $\tilde{q}(\tilde{{\alpha}})$ with a linear equality constraint $\tilde{{\alpha}}^\top {1} = 1$, we use the constrained conjugate gradient method \citep{gould2001solution} as in Algorithm~\ref{alg:cg}. If its output $\tilde{ \alpha}^*$ results in a higher objective value than the input $\tilde{ \alpha}$, then $\tilde{ \alpha}^*$ is a saddle point of the quadratic function instead of a local minimizer. In this case we can restart projected gradient at $\tilde{ \alpha}$ until the working set $\mathcal{W}$ changes. 
\begin{algorithm}
	\begin{algorithmic}[1]
		\renewcommand{\algorithmicrequire}{\textbf{Input:}}
		\renewcommand{\algorithmicensure}{\textbf{Output:}}
		\Require $\tilde{ H}$ (Hessian), $\tilde{ c}$ (linear coefficient), $\tilde{ \alpha}$ (initial point)
		\Ensure  $\tilde{ \alpha}^*$
		\State $ r = \tilde{ H} \tilde{ \alpha} - \tilde{ c}$, $ g = (m r_1 - \sum_{i=1}^{m} r_i, r_2 - r_1, \dots, r_m - r_1)$, $ p = - g$
		\For {$i= 1, \dots, m$}
		\State $t =  r^\top  g$, $a = t /   p^\top \tilde{ H}  p$
		\If {$t == 0$}
		\State terminate the algorithm
		\EndIf
		\State $\tilde{ \alpha} \gets \tilde{ \alpha} + a  p$
		\State $ r \gets  r + a \tilde{ H}  p$
		\State $ g \gets  (m r_1 - \sum_{i=1}^{m} r_i, r_2 - r_1, \dots, r_m - r_1)$
		\State $u =  r^\top  g / t $
		\State $ p \gets - g + u  p$
		\EndFor
		\State $\tilde{ \alpha}^* = \tilde{ \alpha}$
	\end{algorithmic}
	\caption{Constrained Conjugate Gradient (modified Algorithm \rom{2} in \citep{gould2001solution})}
	\label{alg:cg}
\end{algorithm}

\subsection{An Active Set Algorithm}
It is important to know whether the current working set $\mathcal{W}$ is the optimal active set $\mathcal{A}^*$, otherwise we would waste computation on conjugate gradient, which will fail on non-optimal active set.  \cite{dembo1984minimization} provided insights on how to switch from gradient projection to conjugate gradient. \cite{hager2006new} further proposed an active set algorithm that exploit clever strategies to switch between projected gradient and conjugate gradient method. Based on these insights, we propose an active set algorithm (Algorithm~\ref{alg:asa}) to efficiently solve QP on standard simplex.

The reason behind line 9 in Algorithm~\ref{alg:asa} is that, if the angle between $\tilde{ g}^P$ and $\tilde{ g}^R$ is small enough (i.e., smaller than $\theta_1$), then the current working set $\mathcal{W}$ is likely to be the optimal active set \citep{hager2006new}, so we stick to $\mathcal{W}$ and adopt $\tilde{ g}^R$, which will not change $\mathcal{W}$ unless the boundary of $\mathcal{W}$ is encountered. If the angle between $\tilde{ g}^P$ and $\tilde{ g}^R$ is large, then we would like to explore the feasible set following $\tilde{ g}^P$, which is closer to the original gradient $d$ than $\tilde{ g}^R$. 

Similarly, in line 22 of Algorithm~\ref{alg:asa}, if the iterations have been in the same working set for a long time (see line 16 -- 18) and the angle between $\tilde{ g}^P$ and $\tilde{ g}^R$ is small enough, then the current working set $\mathcal{W}$ is very likely to be the optimal active set \citep{hager2006new}, so we switch to conjugate gradient method. If the guess is wrong and conjugate gradient fails, we simply switch back to gradient projection, otherwise a stationary  point is found and we terminate the algorithm.
\begin{algorithm}
	\begin{algorithmic}[1]
		\renewcommand{\algorithmicrequire}{\textbf{Input:}}
		\renewcommand{\algorithmicensure}{\textbf{Output:}}
		\Require $ H$ (Hessian), $ c$ (linear coefficient), ${\alpha}$ (initial point), $\epsilon$, $\theta_1$, $\theta_2$; default parameters are $\epsilon=10^{-8}$, $\theta_1 = \nicefrac{\pi}{18}$, $\theta_2=\nicefrac{\pi}{90}$. 
		\Ensure  ${\alpha}^*$
		\State $t=0$, $flag = 1$, $n = |\alpha|$, initialize $\mathcal{W}$ according to ${\alpha}$ \Comment{$n$ is number of variables}
		\While {not converge}
		\State $ d=H\alpha - c$ \Comment{Compute gradient of $q({\alpha})$}
		\State Compute $ g^P$ and $ g^R$ by \eqref{eq:twodirection}
		\State Use Algorithm~\ref{alg:proj grad} to project $ g^P$ and $ g^R$ to get $\tilde{ g}^R$ or $\tilde{ g}^P$
		\If {$\| \tilde{ g}^P \| < \epsilon$}
		\State terminate the algorithm
		\EndIf
		\If {$\langle \tilde{ g}^P, \tilde{ g}^R \rangle < \theta_1$} \Comment{The angle between $\tilde{ g}^P$ and $\tilde{ g}^R$ is small enough}
		\State $ p = \tilde{ g}^R$
		\Else
		\State $ p = \tilde{ g}^P$
		\EndIf
		\State ${\alpha} \gets {\alpha} - u  p$ where $u$ is determined by exact line search
		\State $t \gets t + 1$
		\If {working set $\mathcal{W}$ changes}
		\State $t \gets 0$, $flag \gets 1$
		\EndIf
		\State $ d=H\alpha - c$ \Comment{Compute gradient of $q({\alpha})$}
		\State Compute $ g^P$ and $ g^R$ by \eqref{eq:twodirection}
		\State Use Algorithm~\ref{alg:proj grad} to project $ g^P$ and $ g^R$ to get $\tilde{ g}^R$ or $\tilde{ g}^P$
		\If {$flag == 1$ \textbf{and} $ t > n $ \textbf{and} $ \langle  \tilde{g}^P,  \tilde{g}^R \rangle < \theta_2$}
		\State run Algorithm~\ref{alg:cg} on free variables $\tilde{{\alpha}}$ to get $\tilde{{\alpha}}^*$
		\If {$\tilde{q}(\tilde{{\alpha}}^*) > \tilde{q}(\tilde{{\alpha}})$ \textbf{or} $\tilde{{\alpha}}^* \notin \mathbb{R}_{\geq 0}^m$}
		\State $flag \gets 0$
		\State continue the algorithm with ${\alpha}$ (abandon $\tilde{{\alpha}}^*$)
		\Else 
		\State set the free variables on ${\alpha}$ according to $\tilde{{\alpha}}^*$ and end the algorithm with ${\alpha}$
		\EndIf
		\EndIf
		\EndWhile
	\end{algorithmic}
	\caption{Active Set Algorithm to solve QP on Simplex}
	\label{alg:asa}
\end{algorithm}

\section{Conclusion}
In this paper we propose an efficient algorithm to project the gradient of a function to hyperplane with sign constraints, which can be used to optimize a differentiable function with standard simplex constraints. The algorithm is based on the ordering properties of the elements in optimal solutions. 
Furthermore, we apply the proposed gradient projection method to quadratic programs (QP) with standard simplex constraints, where gradient projection is used to explore the feasible region and, when we believe the optimal active set is identified, constrained conjugate gradient is exploited to accelerate convergence. Specifically, two different directions of gradient projection are used to explore the simplex, namely, the projected gradient and the reduced gradient. We choose one of the two directions according to the angle between the directions. Moreover, we propose two conditions for guessing the optimal active set heuristically. The first condition is that the working set remains unchanged for many iterations, and the second condition is that the angle between the projected gradient and the reduced gradient is small enough. Based on these strategies, a new active set algorithm for solving quadratic programs on standard simplex is proposed.

\bibliography{projection.bib}

\end{document}